\documentclass[10pt,reqno]{amsart}

\usepackage{amscd}
\usepackage{amsfonts}
\usepackage{amsmath}
\usepackage{amssymb}
\usepackage{amsthm}
\usepackage{enumitem}
\usepackage{mathtools}
\usepackage{mathrsfs}
\usepackage{needspace}
\usepackage{setspace}
\usepackage{tikz}
\usepackage{version}

\usepackage{hyperref}

\onehalfspace

\theoremstyle{plain}
\newtheorem{lemma}{Lemma}[section]

  \newtheorem{theorem}[lemma]{Theorem}
  \newtheorem*{theorem*}{Theorem}

  \newtheorem*{fact*}{Fact}
  \newtheorem*{claim*}{Claim}

\theoremstyle{definition}
  \newtheorem{definition}[lemma]{Definition}

\theoremstyle{remark}
  \newtheorem{remark}[lemma]{Remark}

\setenumerate{label=\textup{(}\roman*\textup{)}}

\makeatletter
\newsavebox{\@brx}
\newcommand{\llangle}[1][]{\savebox{\@brx}{\(\m@th{#1\langle}\)}%
  \mathopen{\copy\@brx\kern-0.5\wd\@brx\usebox{\@brx}}}
\newcommand{\rrangle}[1][]{\savebox{\@brx}{\(\m@th{#1\rangle}\)}%
  \mathclose{\copy\@brx\kern-0.5\wd\@brx\usebox{\@brx}}}
\makeatother

\title{Periodic nilsequences and inverse theorems on cyclic groups}
\author{Freddie Manners}
\address{Mathematical Institute, Radcliffe Observatory Quarter, Woodstock Road, Oxford OX2 6GG}
\email{Frederick.Manners@maths.ox.ac.uk}

\begin{document}

\newcommand{\eps}[0]{\varepsilon}

\newcommand{\AAA}[0]{\mathbb{A}}
\newcommand{\CC}[0]{\mathbb{C}}
\newcommand{\EE}[0]{\mathbb{E}}
\newcommand{\FF}[0]{\mathbb{F}}
\newcommand{\NN}[0]{\mathbb{N}}
\newcommand{\PP}[0]{\mathbb{P}}
\newcommand{\QQ}[0]{\mathbb{Q}}
\newcommand{\RR}[0]{\mathbb{R}}
\newcommand{\TT}[0]{\mathbb{T}}
\newcommand{\ZZ}[0]{\mathbb{Z}}

\newcommand{\cA}[0]{\mathscr{A}}
\newcommand{\cB}[0]{\mathscr{B}}
\newcommand{\cC}[0]{\mathscr{C}}
\newcommand{\cH}[0]{\mathscr{H}}
\newcommand{\cG}[0]{\mathscr{G}}
\newcommand{\cK}[0]{\mathscr{K}}
\newcommand{\cM}[0]{\mathscr{M}}
\newcommand{\cN}[0]{\mathscr{N}}
\newcommand{\cP}[0]{\mathscr{P}}
\newcommand{\cS}[0]{\mathscr{S}}
\newcommand{\cX}[0]{\mathscr{X}}
\newcommand{\cY}[0]{\mathscr{Y}}
\newcommand{\cZ}[0]{\mathscr{Z}}

\newcommand{\fg}[0]{\mathfrak{g}}
\newcommand{\fk}[0]{\mathfrak{k}}
\newcommand{\fZ}[0]{\mathfrak{Z}}

\newcommand{\bmu}[0]{\boldsymbol\mu}

\newcommand{\AUT}[0]{\mathbf{Aut}}
\newcommand{\Aut}[0]{\operatorname{Aut}}
\newcommand{\GI}[0]{\operatorname{GI}}
\newcommand{\HK}[0]{\operatorname{HK}}
\newcommand{\HOM}[0]{\mathbf{Hom}}
\newcommand{\Hom}[0]{\operatorname{Hom}}
\newcommand{\Ind}[0]{\operatorname{Ind}}
\newcommand{\Lip}[0]{\operatorname{Lip}}
\newcommand{\LHS}[0]{\operatorname{LHS}}
\newcommand{\RHS}[0]{\operatorname{RHS}}
\newcommand{\Sub}[0]{\operatorname{Sub}}
\newcommand{\id}[0]{\operatorname{id}}
\newcommand{\image}[0]{\operatorname{Im}}
\newcommand{\poly}[0]{\operatorname{poly}}
\newcommand{\trace}[0]{\operatorname{Tr}}
\newcommand{\sig}[0]{\ensuremath{\tilde{\cS}}}
\newcommand{\psig}[0]{\ensuremath{\cP\tilde{\cS}}}
\newcommand{\metap}[0]{\operatorname{Mp}}
\newcommand{\symp}[0]{\operatorname{Sp}}
\newcommand{\HCF}[0]{\operatorname{hcf}}
\newcommand{\LCM}[0]{\operatorname{lcm}}

\newcommand{\Conv}[0]{\mathop{\scalebox{1.5}{\raisebox{-0.2ex}{$\ast$}}}}
\newcommand{\bs}[0]{\backslash}

\newcommand{\heis}[3]{ \left(\begin{smallmatrix} 1 & \hfill #1 & \hfill #3 \\ 0 & \hfill 1 & \hfill #2 \\ 0 & \hfill 0 & \hfill 1 \end{smallmatrix}\right)  }

\newcommand{\uppar}[1]{\textup{(}#1\textup{)}}

\maketitle

\begin{abstract}
  The inverse theorem for the Gowers norms, in the form proved by Green, Tao and Ziegler, applies to functions on an interval $[M]$.  A recent paper of Candela and Sisask requires a stronger conclusion when applied to $N$-periodic functions; specifically, that the corresponding nilsequence should also be $N$-periodic in a strong sense.

  In most cases, this result is implied by work of Szegedy (and Camarena and Szegedy) on the inverse theorem.  This deduction is given in Candela and Sisask's paper.  Here, we give an alternative proof, which uses only the Green--Tao--Ziegler inverse theorem as a black box.  The result is also marginally stronger, removing a technical condition from the statement.

  The proof centers around a general construction in the category of nilsequences and nilmanifolds, which is possibly of some independent interest.
\end{abstract}

\tableofcontents

\section{Introduction}

The inverse theorem for the Gowers norms has a central role in many recent developments in additive combinatorics and related fields.  We briefly recall the statement, as it appears in \cite[Conjecture 1.2]{gi-paper}.

\begin{theorem}
  \label{theorem:gis-gtz}
  Fix an integer $s \ge 0$ and $\delta > 0$.  Then there exists a finite collection $\cM_{s, \delta}$ of tuples $(G, \Gamma, d_G)$, where $G$ is an $s$-step filtered, nilpotent Lie group, $\Gamma$ a lattice in $G$ \uppar{i.e.~a discrete co-compact subgroup}, and $d_G$ a left-invariant Riemannian metric on $G$, such that the following holds:  for any positive integer $N$, and any function $f \colon [N] \to \CC$ with $|f| \le 1$ and $\|f\|_{U^{s+1}[N]} \ge \delta$, there exists
  \begin{enumerate}
    \item a tuple $(G, \Gamma, d_G) \in \cM_{s, \delta}$;
    \item a polynomial map $p \colon \ZZ \to G$;
    \item a function $F \colon G \to \CC$ that is automorphic \footnote{This simply means that $F(\gamma x) = F(x)$ for all $x \in G$ and $\gamma \in \Gamma$, i.e.~$F$ descends to a function $\Gamma \bs G \to \CC$.  One could also use the term ``$\Gamma$-periodic (on the left)''.} with respect to $\Gamma$, bounded in magnitude by $1$ and is $O_{s,\delta}(1)$-Lipschitz with respect to $d_G$;
  \end{enumerate}
  such that the \emph{nilsequence} $F \circ p$ correlates with $f$, i.e.~$|\EE_{x \in [N]} f(x) \cdot \overline{F \circ p(x)}| \gg_{s, \delta} 1$.
\end{theorem}

We will not reproduce definitions of the various terms here, referring the reader to \cite{gi-paper} itself, \cite{higher-order} or the forthcoming \cite{ben-book}.

This result only applies to functions whose domain is an interval in $\ZZ$.  However, in the case $s=1$, which is covered by classical Fourier analysis, similar statements can be made in much more general domains; in particular cyclic groups $\ZZ/N\ZZ$.  Moreover, in this setting the structured functions $F \circ p$ can be taken to be characters on $\ZZ/N\ZZ$, and hence naturally respect the algebraic structure of $\ZZ/N\ZZ$.

For general $s$, in some sense the algebraic content is reflected in the polynomial map $p$.  So, in formulating a version of Theorem \ref{theorem:gis-gtz} for cyclic groups, it makes some sense to require that $p$ respects the algebraic structure of $\ZZ/N\ZZ$.

\begin{definition}
  \label{definition:n-periodic}
  Let $G$ be a (filtered) nilpotent Lie group and $\Gamma$ a lattice in $G$.  We say a polynomial map $p \colon \ZZ \to G$ is \emph{$N$-periodic} if $p(x + N) p(x)^{-1} \in \Gamma$ for all $x$.
\end{definition}
  Informally, we refer to a function of the form $F \circ p \colon \ZZ \to \CC$, for an $N$-periodic polynomial map $p \colon \ZZ \to G$ and some automorphic function $F \colon G \to \CC$, as an \emph{\uppar{$N$-}periodic nilsequence}.

\begin{remark}
  Note that an $N$-periodic nilsequence in this sense should not be confused with the weaker property that $F \circ p$ is periodic \emph{as a function}, i.e.~that $F \circ p(x + N) = F \circ p(x)$ for all $x$.
\end{remark}

We can now state the main result, the $N$-periodic analogue of Theorem \ref{theorem:gis-gtz}.

\begin{theorem}[Main theorem]
  \label{theorem:main-thm}
  Fix an integer $s \ge 1$ and $\delta > 0$.  Then there exists a finite collection $\cM'_{s, \delta}$ of tuples $(G, \Gamma, d_G)$ as in Theorem \ref{theorem:gis-gtz}, such that the following holds:  for any positive integer $N$, and any function $f \colon \ZZ/N\ZZ \to \CC$ with $|f| \le 1$ and $\|f\|_{U^{s+1}(\ZZ/N\ZZ)} \ge \delta$, there exists
  \begin{enumerate}
    \item a tuple $(G, \Gamma, d_G) \in \cM'_{s, \delta}$;
    \item an $N$-periodic polynomial map $p \colon \ZZ \to G$;
    \item a function $F \colon G \to \CC$ that is automorphic with respect to $\Gamma$, bounded in magnitude by $1$ and is $O_{s,\delta}(1)$-Lipschitz with respect to $d_{G}$;
  \end{enumerate}
  such that $F \circ p$ correlates with $f$, i.e.~$|\EE_{x \in \ZZ/N\ZZ} f(x) \cdot \overline{F \circ p(x)}| \gg_{s, \delta} 1$.
\end{theorem}

Such a result finds an application in recent work of Candela and Sisask \cite{candela-sisask}, on convergence (over prime $N \rightarrow \infty$) of the minimum number of solutions to a fixed system of linear equations in a dense subset $A \subseteq \ZZ/N\ZZ$.  The full strength of Definition \ref{definition:n-periodic} is required for this application to succeed.

In that paper, the authors establish Theorem \ref{theorem:main-thm} assuming a technical condition on $N$, as a consequence of Szegedy's approach to the inverse theorem for the Gowers norms (see in particular \cite{szegedi-higher} as well as joint work with Camarena \cite{cs}).

Our approach is to prove the following result, which constructs periodic nilsequences from non-periodic ones.  Assuming this, it is straightforward to deduce Theorem \ref{theorem:main-thm}.

\begin{theorem}
  \label{theorem:periodic-lift}
  Suppose we are given an ($s$-step filtered) nilpotent Lie group $G$, a lattice $\Gamma$ in $G$ and a left-invariant Riemannian metric $d_G$, as well as a polynomial map $p \colon \ZZ \to G$. Also suppose $F \colon G \to \CC$ is an automorphic, $K$-Lipschitz function bounded in magnitude by $1$.  Further, suppose $\phi \colon \RR/\ZZ \to [0,1]$ is a smooth function supported on $[0,1/2]$, and $N$ is a positive integer.
  
  Then there exists:
  \begin{enumerate}
    \item a \uppar{canonical} $s$-step filtered nilpotent Lie group $\tilde{G}$ and lattice $\tilde{\Gamma}$ in $\tilde{G}$, which depend only on $G$ and $\Gamma$;
    \item a \uppar{canonical} $N$-periodic polynomial map $\tilde{p} \colon \ZZ \to \tilde{G}$ which depends only on $p$ \uppar{and $G, \Gamma$}; and
    \item an $O_{K, \Gamma, G, \phi}(1)$-Lipschitz function $\tilde{F} \colon \tilde{G} \to \CC$, automorphic with respect to $\tilde{\Gamma}$ and bounded in magnitude by one;
  \end{enumerate}
  such that
  \[
    \tilde{F}(\tilde{p}(x)) = \phi(x / N) \cdot F(p(x \bmod N))
  \]
  for all $x$, where $x \bmod N$ refers to the representative in $\{0, \dots, N-1\}$.
\end{theorem}

\section{Deduction of Theorem \ref{theorem:main-thm}}
\label{sec:deduction}

We record the following easy lemma.
\begin{lemma}
  \label{lemma:norms-agree}
  Suppose $f \colon \ZZ/N\ZZ \to \CC$ is supported on an interval $J$ in $\ZZ/N\ZZ$ satisfying $N/20 \le |J| \le N/2$.  Write $f' \colon J \to \CC$ for the restriction of $f$ to that interval.  Then for any $s \ge 1$,
  \[
    \|f'\|_{U^{s+1}(J)} = C \|f\|_{U^{s+1}(\ZZ/N\ZZ)}
  \]
  for some constant $C$ independent of $f$, satisfying $C = \Theta_s(1)$.
\end{lemma}
\begin{proof}
  By translation-invariance we may assume $J = \{0, \dots, |J| - 1\}$.  It suffices to prove that any parallelepiped $(x + h \cdot \omega)_{\omega \in \{0,1\}^{s+1}}$ in $\ZZ/N\ZZ$ that is entirely contained in $J$ is also a parallelepiped with respect to $\ZZ$ (after embedding $\{0, \dots, |J|-1\}$ in $\ZZ$ in the obvious way).  The constant $C$ then arises purely from the normalization constant in the definition of a Gowers norm on an interval and can be ignored.

  It suffices to check the case $s=1$, since any configuration $(x_\omega)$ is a parallelepiped if and only if all its $2$-dimensional ``faces'' are parallelepipeds.  In other words, we want to know that if $x,y,z,w \in \{0,\dots,|J|-1\}$ and $x - y - z + w \equiv 0 \pmod{N}$ then $x - y - z + w = 0$; but this is clear as $-2 (|J| - 1) \le x - y - z + w \le 2 (|J| - 1)$.
\end{proof}

We proceed to the proof.
\begin{proof}[Proof of Theorem \ref{theorem:main-thm} assuming Theorem \ref{theorem:periodic-lift}]
  We are free to assume, to avoid tedious issues, that $N$ is not too small (say $N \ge 100$) since the result is trivial for small $N$ (using, say, the $U^2$ case).

  We choose $20$ closed intervals $I_0, \dots, I_{19}$ in $\RR/\ZZ$, each of width exactly $1/10$, whose interiors cover $\RR/\ZZ$.  For sake of argument we can fix $I_m = [m/20, m/20 + 1/10] \pmod{1}$.  Now choose a smooth partition of unity $(\rho_m)$ on $\RR/\ZZ$ adapted to $I_m$.  The functions $\rho_m$ descend to a partition of unity $(\phi_m)$ on $\ZZ/N\ZZ$ in the obvious way, i.e.~by $\phi_m \colon x \mapsto \rho_m(x/N)$.  Then, each such $\phi_m$ is supported on an interval of length either $\lfloor N/10 \rfloor$ or $\lceil N/10 \rceil$; call this interval $J_m$.

  We deduce that
  \begin{align*}
    \|f\|_{U^{s+1}(\ZZ/N\ZZ)} &= \left\|\sum_m \phi_m \cdot f \right\|_{U^{s+1}(\ZZ/N\ZZ)}  \\
     &\le \sum_m \|\phi_m \cdot f\|_{U^{s+1}(\ZZ/N\ZZ)}  \\
     &= C \sum_m \|\phi_m \cdot f\|_{U^{s+1}(J_m)}
  \end{align*}
  where we have used the triangle inequality for $U^{s+1}$ and Lemma \ref{lemma:norms-agree}.

  Hence, for some $m$ we have $\|\phi_m \cdot f\|_{U^{s+1}(J_m)} \gg_s \delta$.  Translating everything if necessary, we can assume that $\rho_m$ is supported on $[0,1/2]$ and $J_m \subseteq \{0, \dots, N/2\}$.

  We apply Theorem \ref{theorem:gis-gtz} to the function $\phi_m \cdot f$ on $J_m$ to obtain $(G, \Gamma, d_G) \in \cM_{s, \Omega_s(\delta)}$ and a nilsequence $\psi = F \circ p$, where $p \colon \ZZ \to G$ is a polynomial map, such that $\psi$ correlates with $\phi_m \cdot f$, i.e.~$|\EE_{x \in J_m} (\phi_m \cdot f)(x) \overline{\psi(x)} | \gg_{s, \delta}(1)$.

  Now apply Theorem \ref{theorem:periodic-lift} to the nilsequence $\psi$ and the smooth cut-off $\rho_m$.  We find that $\psi' \colon x \mapsto \rho_m(x/N) \psi(x \bmod N)$ is an $N$-periodic nilsequence with respect to some $\tilde{G}, \tilde{\Gamma} $ (depending only on $G, \Gamma$).

  Defining $\cM'_{s, \delta}$ to be those $(\tilde{G}, \tilde{\Gamma}, d_{\tilde{G}})$ obtained by applying Theorem \ref{theorem:periodic-lift} to elements of $\cM_{s, \Omega_s(\delta)}$, and observing that 
  \begin{align*}
    \left| \EE_{x \in \ZZ/N\ZZ} f(x) \overline{\psi'(x)} \right| &= \left| \EE_{x \in \{0, \dots, N-1\}} f(x) \overline{\phi_m(x) \psi(x)} \right| \\ &= \left| \EE_{x \in J_m} (\phi_m \cdot f)(x) \overline{\psi(x)} \right| \gg_{s,\delta} 1\ ,
  \end{align*}
  this completes the proof.
\end{proof}

\section{The proof of Theorem \ref{theorem:periodic-lift}}
\label{sec:general}

The construction we will use is very closely related to that described in \cite[Appendix C]{gi-paper}, which in turn is based on an argument of Furstenberg \cite[page 31]{furst}.  The goal there was rather different to ours; namely, to show that any polynomial nilsequence can be realized as a linear one, at the expense of augmenting $G$, $\Gamma$ to some $\tilde{G}$, $\tilde{\Gamma}$.  Both problems require one to ``invent'' the Heisenberg group, or a variant of it, in the special case of the input $G = \RR$, $\Gamma = \ZZ$; and it turns out that their general cases are also closely related.

Our proof therefore recaps that construction, with only a few (albeit significant) modifications.  We quote a number of facts from \cite[Appendix C]{gi-paper} without reproducing the proofs.

\begin{proof}[Proof of Theorem \ref{theorem:periodic-lift}]

    We consider the nilpotent Lie group of polynomial maps $\poly(\ZZ, G)$ with pointwise multiplication.  There is an isomorphism $\poly(\ZZ, G) \cong \poly(\RR, G)$ given by restriction (see \cite[Appendix C]{gi-paper}).  We also consider the subgroup $\poly(\ZZ, \Gamma)$.  By \cite[Lemma C.1]{gi-paper} the latter is discrete and co-compact.

  We also define the shift action
  \begin{align*}
    T \colon \RR \times \poly(\RR, G) &\to \poly(\RR, G) \\
                                 (t, p) &\mapsto p(\cdot + t)
  \end{align*}
  and thereby the semi-direct product $\tilde{G} = \poly(\RR, G) \rtimes_T \RR$.  To avoid confusion we note that by this ordering we mean to imply a group operation
  \[
    (g, t) \ast (g', t') = (g \cdot t(g'), t + t') \ .
  \]
  
  This has a subgroup $\tilde{\Gamma} = \poly(\ZZ, \Gamma) \rtimes_T \ZZ$ (since the action of $T$ respects $\poly(\ZZ, \Gamma)$).  The latter is discrete and co-compact in $\tilde{G}$.  Indeed, every right coset $\tilde{\Gamma} x$ has a (unique) representative $(\poly(\ZZ, \Gamma) g, t)$ where $t \in [0,1)$.

    We can also be explicit about the filtration on $\tilde{G}$, which is $\tilde{G}_0 = \tilde{G}$, $\tilde{G}_1 = \poly(\RR, G^{+1}) \rtimes_T \RR$ and $\tilde{G}_i = \poly(\RR, G^{+i}) \rtimes_T \{0\}$ for $i \ge 2$.  Here $G^{+i}$ denotes the group $G_i$ with the shifted filtration $G^{+i}_j = G_{i+j}$.  Note in particular that $\tilde{G}_{s+1}$ is trivial (noting $s \ge 1$).

  Let $q \in \poly(\RR, G)$ be a rescaled version of $p$, by $q(x) = p(x N)$.  Now we define
  \begin{align*}
    \tilde{p} \colon \ZZ &\to \tilde{G} \\
    x &\mapsto (\id_G, x/N) \ast (q, 0)
  \end{align*}
  which is certainly a polynomial map, and moreover is $N$-periodic since
  \begin{align*}
    \tilde{p}(x + N) &= (\id_G, x/N + 1) \ast (q, 0) \\
                                      &= (\id_G, 1) \ast (\id_G, x/N) \ast (q, 0) \\
                                    &\in \tilde{\Gamma}\, \tilde{p}(x) \ .
  \end{align*}
  Finally, we define $\tilde{F}$ on the fundamental domain $(\poly(\ZZ, \Gamma) \bs \poly(\RR, G)) \rtimes_T [0, 1)$ by
  \[
    \tilde{F}(\poly(\ZZ, \Gamma)\, x, t) = \phi(t) \cdot F(\Gamma\, x(0))
  \]
  and extend periodically by $\tilde{\Gamma}$.  In other words,
  \[
    \tilde{F}(\poly(\ZZ, \Gamma)\, x, t) = \phi(\{t\}) \cdot F(\Gamma\, x(-\lfloor t \rfloor))
  \]
  for general $t$.

  It is now straightforward to check that
  \begin{align*}
    \tilde{F}(\tilde{p}(x)) &= \tilde{F}(q(\cdot + x/N), x/N) \\
                            &= \phi(\{x/N\}) \cdot F(q(x/N - \lfloor x/N \rfloor)) \\
                            &= \phi(\{x/N\}) \cdot F(p(x \bmod N))
  \end{align*}
  as required.
  \\[\baselineskip]
  There remains only the statement that $\tilde{F}$ is $O_{K, G, \Gamma, \phi}(1)$-Lipschitz.  For this to makes sense, we need to equip $\poly(\RR, G)$ and $\tilde{G}$ with left-invariant Riemannian metrics $d_{\poly(\RR, G)}$, $d_{\tilde{G}}$, which define corresponding metrics $d_{\poly(\ZZ, \Gamma) \bs \poly(\RR, G)}$, $d_{\tilde{\Gamma} \bs \tilde{G}}$.  There are various ways one could make these choices canonical; ultimately it doesn't matter as any two choices will be bi-Lipschitz on the respective compact manifolds.
  
  It is certainly clear that $\tilde{F}$ is continuous: this is immediate except at points $(t, x)$ where $t \in \ZZ$, but the support properties of $\phi$ guarantee continuity there also.  Note also that the definition of $\tilde{F}$ depends only on $F$, $G$ and $\Gamma$, not $p$ or $N$.

  Using compactness of $\tilde{\Gamma} \bs \tilde{G}$, it is routine to adapt these observations to show $\tilde{F}$ is Lipschitz and to obtain the quantitative statement about the Lipschitz constant.

\end{proof}

\section{An example}
\label{sec:example}

We consider the case where $s=2$ and the original nilsequence $f$ is the function
\begin{align*}
  f \colon \ZZ &\to \CC \\
  x &\mapsto e(\alpha x^2)
\end{align*}
for some real number $\alpha$.  Note that this is trivially a (polynomial) nilsequence by composing the polynomial map $p \colon \ZZ \to \RR$, $p(x) = \alpha x^2$ with the automorphic function $F(t) = e(t)$.

In general, $p$ fails to be $N$-periodic unless $\alpha \in \frac{1}{N} \ZZ$.   Moreover, $f$ can fail to correlate significantly with any $N$-periodic phase quadratic: i.e.~(to prove Theorem \ref{theorem:main-thm}) it would not suffice just to alter the value of $\alpha$.

Theorem \ref{theorem:periodic-lift} requires us (essentially) to consider the function $f'$ obtained by repeating $f$ $N$-periodically, i.e.~$f'(x) = e(\alpha N^2 \{x / N\}^2)$.  This is a bracket quadratic, and hence general machinery due to Bergelson and Leibman \cite{bergelson-leibman} tells us that it is -- essentially -- a nilsequence.  With such a simple example we can make this very explicit.  To simplify matters, we consider the case where $\alpha = a / N^2$ for some integer $a$; these examples \emph{are} fairly dense in the whole space.

Let
\[
  G = \heis{\RR}{\RR}{\RR};\ \ \Gamma = \heis{\ZZ}{\ZZ}{\ZZ}
\]
so $G$ is the usual Heisenberg group and $\Gamma$ the standard lattice.  We abbreviate the matrix
$
  \heis{x}{y}{z}
$
to $(x,y,z)$.  Now define
\begin{align*}
  p \colon \ZZ &\to G \\
t &\mapsto (t / N, 2 a t / N, a t^2 / N^2)
\end{align*}
which is a polynomial map.  Moreover, it is $N$-periodic, since
\[
  p(t + N) = \heis{t / N  + 1}{2 a t / N + 2 a}{a t^2/N^2 + 2 a t / N + a} = \heis{1}{2 a}{a} p(t) \in \Gamma p(t) \ .
\]
A fundamental domain for $\Gamma \bs G$ is given by elements $(x,y,z) \in G$ with $x,y,z \in [0,1)$, and the map sending a point in $G$ to the corresponding point in the fundamental domain is
$
  (x,y,z) \mapsto (\{x\}, \{y\}, \{z - y \lfloor x \rfloor \})
$,
so we define $F \colon G \to \CC$ by $F(x,y,z) = \phi(x) e(z)$ for points in the fundamental domain, and extend $\Gamma$-periodically.  Here $\phi \colon \RR / \ZZ \to \RR_{\ge 0}$ is our smooth function supported on $[0,1/2]$.  In other words,
$
  F(x,y,z) = \phi(\{x\}) e(z - y \lfloor x \rfloor) 
$.
It is now straightforward to check that $F \circ p(t) = \phi(t/N) e(a t^2 / N^2)$ if $0 \le t < N$, which together with the $N$-periodicity above gives the result.
\\[\baselineskip]
It turns out that this construction is \emph{not} exactly what comes out of the general machinery of Section \ref{sec:general}.  However, the two are closely, albeit slightly subtly, related.  We will not run through the construction in detail, but sketch some of the features.

In the Heisenberg example above, the group $G$ is $2$-step nilpotent with a $2$-step filtration.  The group $\tilde{G}$ from the general construction turns out to be $3$-step nilpotent \emph{as a group}, even though we were careful to give it a $2$-step filtration.

This apparent paradox is a consequence of the fact that the filtration of $\tilde{G}$ is not \emph{proper}, i.e.~$\tilde{G}_0 \ne \tilde{G}_1$.  Hence all the $3$-step behaviour of $\tilde{G}$ is ``hidden'' in $\tilde{G}_0$ where it does not significantly affect the filtration.\footnote{A group with an improper $s$-step filtration need not even be nilpotent in general.}

Although it may seem illegal to use a $3$-step nilpotent group when $s=2$, this is in fact compatible with the definitions of \cite{gi-paper}.  Indeed, the same phenomenon arises when applying the construction in Appendix C of that paper.

In the special case that the image of $\tilde{p}$ is contained in $\tilde{G}_1$, we can restrict everything to $\tilde{G}_1$ equipped with the proper filtration $\tilde{G}_1 = \tilde{G}_1 \supseteq \tilde{G}_2 \supseteq \{\id\}$.  Hence $\tilde{G}_1$ is a $2$-step nilpotent group.  Moreover, it turns out that this special case occurs whenever $\alpha = a / N^2$ as required above.  We also find that $\tilde{G}_1$ is exactly the Heisenberg group, and indeed this process yields exactly the explicit construction we have just described.

For general $\alpha$ we do not have that the image of $\tilde{p}$ is contained in $\tilde{G}_1$; but it is contained in some \emph{coset} of $\tilde{G}_1$.  So, by shifting everything by a small element of $\tilde{G}$ we can move to the previous case and again identify $\tilde{F} \circ \tilde{p}$ with a Heisenberg nilsequence.

In fact, it is generally true that any ``improper'' nilsequence in the above sense can be identified with a ``proper'' one, but we will not prove such a result.

\bibliography{master}{}
\bibliographystyle{halpha}

\end{document}